\documentclass[a4paper,reqno]{amsart}
\usepackage{amssymb, amsmath, amscd}
\def\XIx\langle#1\rangle{h(#1)}
\def\gronko{\vphantom{\vrule height 11pt}}
\newtheorem{theorem}{Theorem}[section]
\newtheorem{definition}[theorem]{Definition}
\newtheorem{lemma}[theorem]{Lemma}

\newtheorem{remark}[theorem]{Remark}

\def\id{\operatorname{id}}

\def\mapright#1{\smash{\mathop{\longrightarrow}\limits\sp{#1}}}
\makeatletter
 \@addtoreset{equation}{section}
\makeatother
\begin{document}
\title[Affine K\"ahler curvature tensors]{The structure of the space of affine K\"ahler curvature tensors as a complex module}
\author{M. Brozos-V\'{a}zquez}
\address{Escola Polit\'ecnica Superior, Department of Mathematics, University of A Coru\~na, Spain}
\email{mbrozos@udc.es}
\author{P. Gilkey}
\address{Mathematics Department, University of Oregon, Eugene OR 97403, USA}
\email{gilkey@uoregon.edu}
\author{S. Nik\v cevi\'c}
\address{Mathematical Institute, Sanu,
Knez Mihailova 36, p.p. 367,
11001 Belgrade,
Serbia.}
\email{stanan@mi.sanu.ac.rs}
\begin{abstract}{We use results of Matzeu and Nik\v cevi\'c to decompose the space of affine K\"ahler curvature tensors as a direct sum of irreducible
modules in the complex setting.\\MSC 2010: 53B05, 15A72, 53B35.}\end{abstract}
\maketitle

\section{Introduction}
\subsection{Curvature decompositions}
We begin by giving a brief history and overview of the theory of curvature decompositions to put the main result of this paper in the proper setting.
Such decompositions are central to the theory of modern differential geometry. Consequently, the subject is a vast one and we can only sketch a few
of the highlights. The decompositions in general stabilize; there is a crucial dimension $m_0$ so that if the dimension $m$ exceeds $m_0$ then the
number of summands is constant; one obtains the decomposition in lower dimensions by setting certain of the summands to $\{0\}$.
Singer and Thorpe
\cite{ST69} showed that the space $\mathfrak{R}$ of Riemann curvature tensors has 3 irreducible components under the action of the orthogonal group
$\mathcal{O}$ in dimension $m\ge4$; these are the space of Weyl conformal curvature tensors, the space of trace free Ricci tensors, and the space of
constant sectional curvature tensors. There are only 2 components in dimension $3$ and only 1 component in dimension $2$. Tricerri and Vanhecke
\cite{TV81} gave a similar decomposition of $\mathfrak{R}$ in the almost Hermitian setting; the appropriate structure group there is the unitary group
$\mathcal{U}^\star$ and there are 10 irreducible unitary modules comprising the decomposition in dimension $m\ge8$; if $m=6$, then there are 9 summands and
if $m=4$, then there are 7 summands in the decomposition. If one assumes that the complex structure involved is in fact integrable, Gray \cite{gray}
showed one of the components does not appear so there are 9 irreducible unitary modules in the decomposition in the context of Hermitian geometry if
$m\ge8$, $8$ if $m=6$, and $6$ if $m=4$. K\"ahler geometry remains a field of active investigation in many different contexts
\cite{K10,M09,PSSW11,WZ11}; the Riemannian K\"ahler curvature tensors have 3 factors in their decomposition ($m\ge4)$ as unitary modules. Note that
Sasakian geometry is intimately linked with K\"ahler geometry -- see, for example, the discussion in \cite{C11,GV} -- so odd dimensional phenomena can
also appear in this setting.  De Smedt \cite{S93} showed there are 37 modules in the decomposition of
$\mathfrak{R}$ under the action of the symplectic group in the hyper-Hermitian setting for $m\ge16$ (the number drops to $36$ if $m=12$ and to $32$ if
$m=8$). Hyper-K\"ahler geometry also is being actively studied -- see, for example
\cite{Ca10,GL11,O10}.

Although not a curvature decomposition, the following decomposition is in the same spirit. Let $\nabla\Omega$ be the covariant derivative of the
K\"ahler form on an almost Hermitian manifold. Gray and Hervella
\cite{GH80} showed that
$\nabla\Omega$ can be decomposed into 4 separate components if
$m\ge6$ and 2 components if $m=4$; this gives rise to the celebrated $16=2^4$ classes of almost Hermitian manifolds. We also refer to subsequent
results of Brozos-V\'azquez et al. \cite{BGGH11} in the almost pseudo-Hermitian and in the almost para-Hermitian settings.

Weyl geometry is in a certain sense midway between Riemannian and affine geometry. Higa \cite{H93,H94} decomposed the space
of Weyl curvature tensors into irreducible orthogonal modules; there are $4$ summands if $m\ge4$. We refer to \cite{BNGS06,GNU09,GNS09} for further
details in this regard. Strichartz \cite{S88} decomposed the space of affine curvature tensors as a direct sum of $3$ modules over the general
linear group $\operatorname{GL}$ if
$m\ge3$; we present his result in Theorem \ref{thm-1.1} below. Subsequently, Bokan \cite{B90} decomposed this space as an orthogonal module; there
are 8 summands if $m\ge4$, 6 summands if $m=3$, and $3$ summands if $m=2$. This decomposition is perhaps less natural since an auxiliary inner product
needs to be introduced. Matzeu and Nik\v cevi\'c
\cite{PN91,N89} generalized Bokan's work to decompose the space of K\"ahler affine curvature tensors $\mathcal{K}$ as a unitary module; there are 12
summands in the decomposition if $m\ge6$ and 10 summands in the decomposition if $m=4$. This result will be presented as Theorem
\ref{thm-1.5}. In this present paper, we use Theorem
\ref{thm-1.5} to establish Theorem \ref{thm-1.2} which generalizes Theorem \ref{thm-1.1} to the complex setting; there are 6 summands in the
decomposition for
$m\ge4$.

\subsection{Affine structures}
We now introduce the requisite notation to state the results of \cite{PN91,N89,S88} and the main result of this paper more precisely. An {\it affine
manifold} is a pair
$(M,\nabla)$ where $M$ is a smooth manifold and where $\nabla$ is a torsion free connection on the tangent bundle $TM$. We
refer to
\cite{SSV91} for further information concerning affine geometry. The associated {\it curvature
operator}
$\mathcal{R}\in\otimes^2T^*M\otimes\operatorname{End}(TM)$ is defined by setting:
$$\mathcal{R}(x,y):=\nabla_x\nabla_y-\nabla_y\nabla_x-\nabla_{[x,y]}\,.$$
This tensor satisfies the following identities:
\begin{equation}\label{eqn-1.a}
\begin{array}{ll}
\mathcal{R}(x,y)=-\mathcal{R}(y,x)\text{ and }
\mathcal{R}(x,y)z+\mathcal{R}(y,z)x+\mathcal{R}(z,x)y=0\,.\vphantom{\vrule height 11pt}\end{array}
\end{equation}
It is convenient to work in a purely algebraic context. Let $V$ be a real $m$-dimensional vector space. We say
that $A\in\otimes^2V^*\otimes\text{End}(V)$ is an {\it affine curvature operator} if $A$ has the symmetries
given above in Equation (\ref{eqn-1.a}). Let $\mathfrak{A}$ be the subspace of all such operators.

The natural structure group in this setting is the {\it general linear group} $\operatorname{GL}$. The Ricci
tensor $\rho$ is a $\operatorname{GL}$ equivariant map from $\mathfrak{A}$ to $V^*\otimes V^*$ defined by setting:
$$\rho(x,y):=\operatorname{Tr}\{z\rightarrow\mathcal{R}(z,x)y\}\,.$$
We decompose $\otimes^2V^*=\Lambda^2\oplus S^2$ into the space  of alternating
$2$-tensors $\Lambda^2$  and the space of symmetric $2$-tensors $S^2$. We summarize below the fundamental decomposition
of the space of affine curvature operators $\mathfrak{A}$ under the natural action of the general linear group
\cite{S88}:

\goodbreak\begin{theorem}\label{thm-1.1}
If $m\ge3$, then $\mathfrak{A}\approx\{\mathfrak{A}\cap\ker(\rho)\}\oplus\Lambda^2\oplus S^2$ as a
$\operatorname{GL}$ module
where
$\{\mathfrak{A}\cap\ker(\rho),\Lambda^2,S^2\}$ are inequivalent and irreducible
$\operatorname{GL}$ modules.
\end{theorem}

\subsection{Affine K\"ahler Structures}\label{sect-1.5.AKG}

The triple $(M,J,\nabla)$ is said to be an {\it affine K\"ahler manifold} if $J$ is an almost complex structure on $M$ (i.e. an endomorphism of the
tangent bundle $TM$ so that $J^2=-\id$), if $\nabla$ is a torsion free connection on $TM$, and if $\nabla J=0$; necessarily the complex structure is
integrable in this setting. The curvature operator
$\mathcal{R}$ then satisfies the additional symmetry:
\begin{equation}\label{eqn-1.b}
J\mathcal{R}(x,y)=\mathcal{R}(x,y)J\quad\text{for all}\quad x,y\,.
\end{equation}
We pass to the algebraic context. Let $J$ be a complex structure on a real vector space $V$. We consider the subgroup of all linear maps commuting or
anti-commuting with
$J$:
$$\operatorname{GL}_{\mathbb{C}}^\star=\{\Xi\in\operatorname{GL}:\Xi J=\pm J\Xi\}\,.$$
We set $\chi(\Xi)=\pm1$ to define a $\mathbb{Z}_2$ representation of $\operatorname{GL}_{\mathbb{C}}^\star$ into $\mathbb{Z}_2$. We shall allow into
consideration maps which replace $J$ by $-J$ as the two complex structures $J$ and $-J$ play interchangable roles in many geometric settings; the
group $\operatorname{GL}_{\mathbb{C}}^\star$ is a $\mathbb{Z}_2$ extension of the usual complex general group.

The space
of K\"ahler affine tensors is defined by imposing the K\"ahler identity given in the geometric setting by Equation (\ref{eqn-1.b}), namely:
$$\mathcal{K}:=\{\mathcal{A}\in\mathfrak{A}:\mathcal{A}(v_1,v_2)J=J\mathcal{A}(v_1,v_2)\ \forall v_1,v_2\in V\}\,.$$
$J$ acts by pullback on tensors of all types.
We may decompose $\Lambda^2=\Lambda^2_+\oplus\Lambda^2_-$, $S^2=S^2_+\oplus S^2_-$, and $\mathcal{K}=\mathcal{K}_+\oplus\mathcal{K}_-$
where
\begin{eqnarray*}
&&\mathcal{K}_\pm:=\{\mathcal{A}\in\mathcal{K}:\mathcal{A}(J
v_1,J v_2)=\pm\mathcal{A}(v_1,v_2)\ \forall v_1,v_2\in V\},\\
&&\Lambda^2_\pm:=\{\psi\in\Lambda^2:\psi(J v_1,J v_2)=\pm\psi(v_1,v_2)\ \forall v_1,v_2\in V\},\\
&&S^2_\pm:=\{\phi\in S^2:\phi(J v_1,J v_2)=\pm\phi(v_1,v_2)\ \forall v_1,v_2\in V\}\,.
\end{eqnarray*}
Since $J$ appears an even number of times, these are $\operatorname{GL}_{\mathbb{C}}^\star$ modules and the Ricci tensor defines short exact sequences
of
$\operatorname{GL}_{\mathbb{C}}^\star$ modules:
$$0\rightarrow\mathcal{K}_\pm\cap\ker(\rho)\rightarrow\mathcal{K}_\pm\mapright{\rho}\Lambda_\pm^2\oplus S_\pm^2\rightarrow0\,.$$
It will follow from Lemma \ref{lem-2.2} that this sequence is split in the category of $\operatorname{GL}_{\mathbb{C}}^\star$ modules; the following result
generalizes Theorem~\ref{thm-1.1} to this setting and is the main result of this paper:
\goodbreak\begin{theorem}\label{thm-1.2}
If $m\ge6$, then we have the following isomorphisms decomposing
$\mathcal{K}_\pm$ as the direct sum of irreducible and inequivalent $\operatorname{GL}_{\mathbb{C}}^\star$ modules:
$$\mathcal{K}_\pm\approx\{\mathcal{K}_\pm\cap\ker(\rho)\}
\oplus\Lambda_\pm^2\oplus S_\pm^2\,.$$
\end{theorem}

\begin{remark}\rm  The modules
$\{\mathcal{K}_+\cap\ker(\rho),\mathcal{K}_-\cap\ker(\rho),S_+^2,S_-^2,\Lambda_-^2\}$ have different dimensions and are therefore inequivalent. Since
$S_+^2$ is not isomorphic to $\Lambda_+^2$ as a $\mathcal{U}^\star$ module (see Theorem \ref{thm-1.5} below), the modules appearing in Theorem
\ref{thm-1.2} are inequivalent. If
$m=4$, the same decomposition pertains if we set the module $\mathcal{K}_-\cap\ker(\rho)=\{0\}$ and therefore delete this module from consideration.
\end{remark}

\subsection{The Matzeu-Nik\v cevi\'c decomposition}
The proof we shall give of Theorem~\ref{thm-1.2} rests on results of \cite{PN91,N89}. We assume given an auxiliary positive definite
inner product
$\langle\cdot,\cdot\rangle$ on $V$ so that $J^*\langle\cdot,\cdot\rangle=\langle\cdot,\cdot\rangle$; the triple $(V,\langle\cdot,\cdot\rangle,J)$ is
said to be a {\it Hermitian vector space}. The orthogonal and unitary groups are then defined by setting:
$$
\mathcal{O}:=\{T\in\operatorname{GL}:T^*\langle\cdot,\cdot\rangle=\langle\cdot,\cdot\rangle\}\quad\text{and}\quad
\mathcal{U}^\star:=\mathcal{O}\cap\operatorname{GL}_{\mathbb{C}}^\star\,.
$$
We use the metric to raise and lower indices. We may now regard:
\begin{eqnarray*}
&&\mathcal{K}:=\{A\in\mathfrak{A}:A(x,y,z,w)=A(x,y,Jz,Jw)\},\\
&&\mathcal{K}_\pm:=\{A\in\mathcal{K}:A(Jx,Jy,z,w)=\pm A(x,y,z,w)\}\,.
\end{eqnarray*}
The decomposition of $\mathcal{K}$ as a unitary module is given in \cite{PN91,N89}; it extends easily to give a $\mathcal{U}^\star$ module
decomposition as well. We first introduce some auxiliary notation:
\begin{definition}
\rm\ Let $(V,\langle\cdot,\cdot\rangle,J)$ be a Hermitian
vector space. Let $\{e_i\}$ be an orthonormal basis for $V$. Adopt the {\it Einstein convention} and sum over repeated indices to define:.
\begin{enumerate}
\smallbreak\item $\rho_{13}(A)(x,y)=A(e_i,x,e_i,y)$ and $\rho(A)(x,y)=A(e_i,x,y,e_i)$.
\smallbreak\item $\Omega(x,y):=\langle x,Jy\rangle$.
\smallbreak\item $S_{0,+}^2:=\{\phi\in S_+^2:\phi\perp\langle\cdot,\cdot\rangle\}$ and $\Lambda_{0,+}^2:=\{\psi\in\Lambda_+^2:\psi\perp\Omega\}$.
\smallbreak\item $W_9:=\{A\in\mathcal{K}_+:A(x,y,z,w)=-A(x,y,w,z)\}\cap\ker(\rho)$.
\smallbreak\item $W_{10}:=\{A\in\mathcal{K}_+:A(x,y,z,w)=A(x,y,w,z)\}\cap\ker(\rho)$.
\smallbreak\item $W_{11}:=\mathcal{K}_+\cap W_{9}^\perp\cap
W_{10}^\perp\cap\ker(\rho_{13})\cap\ker(\rho)$.
\smallbreak\item $W_{12}:=\mathcal{K}_-\cap\ker(\rho)$, $\tau:=A(e_i,e_j,e_j,e_i)$, and
   $\tau_J:=\varepsilon^{il}\varepsilon^{jk} A(e_i,J e_j,e_k,e_l)$.
\end{enumerate}\end{definition}

\goodbreak\begin{theorem}\label{thm-1.5}
Let $m\ge6$. We have decompositions of the following modules as the direct sum of
irreducible and inequivalent $\mathcal{U}^\star$ modules:
\begin{eqnarray*}
&&\mathcal{K}\approx
\mathbb{R}\oplus\chi\oplus2\cdot
S_{0,+}^2\oplus2\Lambda_{0,+}^2\oplus\Lambda_-^2\oplus S_-^2\oplus W_9\oplus W_{10}
\oplus W_{11}\oplus W_{12},\\
&&\mathcal{K}_+\approx\mathbb{R}\oplus\chi\oplus 2S_{0,+}^2\oplus2\Lambda_{0,+}^2\oplus W_9\oplus W_{10}\oplus W_{11},\\
&&\mathcal{K}_+\cap\ker(\rho)\approx S_{0,+}^2\oplus\Lambda_{0,+}^2\oplus W_9\oplus W_{10}\oplus W_{11},\\
&&\mathcal{K}_-\approx \Lambda_-^2\oplus S_-^2\oplus W_{11},\\
&&\mathcal{K}_-\cap\ker(\rho)\approx W_{11}\,.
\end{eqnarray*}
\end{theorem}

\begin{remark}\rm We note that $\mathcal{K}_-\cap\ker(\rho)=W_{12}$ is an irreducible $\mathcal{U}^\star$ module. The decomposition of
Theorem \ref{thm-1.5} is also into irreducible $\mathcal{U}$ modules. However, $S_{0,+}^2$ is isomorphic to
$\Lambda_{0,+}^2$ as a $\mathcal{U}$ module and $W_9$ is isomorphic to $W_{10}$ as a $\mathcal{U}$ module. The corresponding
decompositions if $m=4$ are obtained by setting $W_{11}=W_{12}=\{0\}$.
\end{remark}

\subsection{Outline of the paper} In Section \ref{sect-2}, we shall construct a $\operatorname{GL}_{\mathbb{C}}^\star$ splitting of the map defined by the
Ricci tensor $\rho$ from $\mathcal{K}$ to $\otimes^2V^*$. We use this splitting together with Theorem \ref{thm-1.5} to reduce the proof of Theorem
\ref{thm-1.2} to the assertion that $\mathcal{K}_+\cap\ker(\rho)$ is an irreducible $\operatorname{GL}_{\mathbb{C}}^\star$ module. In Section \ref{sect-3},
we examine $\rho_{13}$ and construct the orthogonal projectors from $\mathcal{K}_+$ to the subspaces of $\mathcal{K}_+\cap\ker(\rho)$ which are
isomorphic to $S_{0,+}^2$ and $\Lambda_{0,+}^2$ in Theorem \ref{thm-1.5}. In Section \ref{sect-4}, we use the conjugate tensor to examine the
orthogonal projectors on the subspaces $W_9$, $W_{10}$, and $W_{11}$ of Theorem \ref{thm-1.5}. In Section \ref{sect-5}, we complete the proof of
Theorem \ref{thm-1.5} by showing $\mathcal{K}_+\cap\ker(\rho)$ is an irreducible $\operatorname{GL}_{\mathbb{C}}^\star$ module.

\section{The geometry of $\rho$}\label{sect-2}
The Ricci tensor defines a $\operatorname{GL}_{\mathbb{C}}^\star$ module morphism
$\rho:\mathcal{K}\rightarrow\otimes^2V^*$
that restricts to $\operatorname{GL}_{\mathbb{C}}^\star$ module morphisms from $\mathcal{K}_\pm$ to $\Lambda_\pm^2\oplus S_\pm^2$.
In this section, we shall construct a $\operatorname{GL}_{\mathbb{C}}^\star$ module morphism splitting of $\rho$. We
first introduce some additional notation:
\begin{definition}
\rm Let $J$ be a complex structure on $V$. For $\phi_1\in S_+^2$,  $\phi_2\in S_-^2$, $\phi_3\in\Lambda_+^2$, and $\phi_4\in\Lambda_-^2$ define:
\medbreak\noindent
$(\sigma_1\phi_1)(x,y)z:=\phi_1(x,z)y-\phi_1(y,z)x-\phi_1(x,J z)J y+\phi_1(y,J z)J x-2\phi_1(x,J y)J z$.
\medbreak\noindent
$(\sigma_2\phi_2)(x,y)z:=\phi_2(x,z)y-\phi_2(y,z)x-\phi_2(x,J z)J y+\phi_2(y,J z)J x$.
\medbreak\noindent
$(\sigma_3\phi_3)(x,y)z:=\phi_3(x,z)y-\phi_3(y,z)x+2\phi_3(x,y)z-\phi_3(x,J z)J y+\phi_3(y,J z)J x$,
\medbreak\noindent
{$(\sigma_4\phi_4)(x,y)z:=\phi_4(x,z)y-\phi_4(y,z)x+2\phi_4(x,y)z$}
%\smallbreak\qquad
$-\phi_4(x,J z)J y+\phi_4(y,J z)J x$\smallbreak$- 2\phi_4(x,J y)J z$.
\end{definition}
Since $J$ appears an even number of times, these are $\operatorname{GL}_{\mathbb{C}}^\star$ module morphisms.

\begin{lemma}\label{lem-2.2}
\
\begin{enumerate}
\item If $\phi_1\in S_+^2$, then $\sigma_1\phi_1\in\mathcal{K}_+$ and
$\rho\,\sigma_1\phi_1=-(m+2)\phi_1$.
\smallbreak\item If $\phi_2\in S_-^2$, then $\sigma_2\phi_2\in\mathcal{K}_-$ and
$\rho\,\sigma_2\phi_2=(2-m)\phi_2$.
\smallbreak\item If $\phi_3\in\Lambda_+^2$, then $\sigma_3\phi_3\in\mathcal{K}_+$ and
$\rho\,\sigma_3\phi_3=-(m+2)\phi_3$.
\smallbreak\item If $\phi_4\in\Lambda_-^2$, then
$\sigma_4\phi_4\in\mathcal{K}_-$ and
$\rho\,\sigma_4\phi_4={-(2+m)}\phi_4$.
\end{enumerate}
\end{lemma}

\begin{proof}
We begin with some basic parity observations:
\medbreak\qquad
$\phi_1(x,J y)=\phi_1(J x,J J y)=-\phi_1(J x,y)$,
\smallbreak\qquad
$\phi_2(x,J y)=-\phi_2(J x,J J y)=\phi_2(J x,y)$,
\smallbreak\qquad
$\phi_3(x,J y)=\phi_3(J x,J J y)=-\phi_3(J x,y)$,
\smallbreak\qquad
$\phi_4(x,J y)=-\phi_4(J x,J J y)=\phi_4(J x,y)$.
\medbreak\noindent
It now follows that the tensors $\{\sigma_1\phi_1, \sigma_2\phi_2,\sigma_3\phi_3,\sigma_4\phi_4\}$
 are anti-symmetric in the first two
arguments. We verify that the Bianchi identity is satisfied by these tensors and therefore that they belong to $\mathfrak{A}$ by computing:
\medbreak\quad $(\sigma_1\phi_1)(x,y)z+(\sigma_1\phi_1)(y,z)x+(\sigma_1\phi_1)(z,x)y$
\smallbreak\qquad
$=\phi_1(x,z)y-\phi_1(y,z)x-\phi_1(x,J z)J y+\phi_1(y,J z)J x-2\phi_1(x,J y)J z$
\smallbreak\qquad
$\phantom{}+\phi_1(y,x)z-\phi_1(z,x)y-\phi_1(y,J x)J z+\phi_1(z,J x)J y-2\phi_1(y,J z)J x$
\smallbreak\qquad
$\phantom{}+\phi_1(z,y)x-\phi_1(x,y)z-\phi_1(z,J y)J x+\phi_1(x,J y)J z-2\phi_1(z,J x)J y=0$,
\smallbreak\quad
$(\sigma_2\phi_2)(x,y)z+(\sigma_2\phi_2)(y,z)x+(\sigma_2\phi_2)(z,x)y$
\smallbreak\qquad
$=\phi_2(x,z)y-\phi_2(y,z)x-\phi_2(x,J z)J y+\phi_2(y,J z)J x$
\smallbreak\qquad
$\phantom{}+\phi_2(y,x)z-\phi_2(z,x)y-\phi_2(y,J x)J z+\phi_2(z,J x)J y$
\smallbreak\qquad
$\phantom{}+\phi_2(z,y)x-\phi_2(x,y)z-\phi_2(z,J y)J x+\phi_2(x,J y)J z=0$,
\smallbreak\quad
$(\sigma_3\phi_3)(x,y)z+(\sigma_3\phi_3)(y,z)x+(\sigma_3\phi_3)(z,x)y$
\smallbreak\qquad
$=\phi_3(x,z)y-\phi_3(y,z)x+2\phi_3(x,y)z-\phi_3(x,J z)J y+\phi_3(y,J z)J x$
\smallbreak\qquad
$\phantom{}+\phi_3(y,x)z-\phi_3(z,x)y+2\phi_3(y,z)x-\phi_3(y,J x)J z+\phi_3(z,J x)J y$
\smallbreak\qquad
$\phantom{}+\phi_3(z,y)x-\phi_3(x,y)z+2\phi_3(z,x)y-\phi_3(z,J y)J x+\phi_3(x,J y)J z=0$,
\smallbreak\quad
$(\sigma_4\phi_4)(x,y)z+(\sigma_4\phi_4)(y,z)x+(\sigma_4\phi_4)(z,x)y$
\smallbreak\qquad
$=\phantom{}\phi_4(x,z)y-\phi_4(y,z)x+2\phi_4(x,y)z$
\smallbreak\qquad
$\phantom{}+\phi_4(y,x)z-\phi_4(z,x)y+2\phi_4(y,z)x$
\smallbreak\qquad
$\phantom{}+\phi_4(z,y)x-\phi_4(x,y)z+2\phi_4(z,x)y$
\smallbreak\qquad
$\phantom{}-\phi_4(x,J z)J y+\phi_4(y,J z)J x- 2\phi_4(x,J y)J z$
\smallbreak\qquad
$\phantom{}-\phi_4(y,J x)J z+\phi_4(z,J x)J y- 2\phi_4(y,J z)J x$
\smallbreak\qquad
$\phantom{}-\phi_4(z,J y)J x+\phi_4(x,J y)J z- 2\phi_4(z,J x)J y=0$.
\medbreak
We verify these endomorphisms commute with $J$ and belong to $\mathcal{K}$ by comparing:
\medbreak\quad
$(\sigma_1\phi_1)(x,y)J z=\phi_1(x,J z)y-\phi_1(y,J z)x-\phi_1(x,J J z)J y$
\smallbreak\qquad
$+\phi_1(y,J J z)J x-2\phi_1(x,J y)J J z$,
\smallbreak\quad
$(J\sigma_1\phi_1)(x,y)z=\phi_1(x,z)J y-\phi_1(y,z)J x-\phi_1(x,J z)J J y$
\smallbreak\qquad
$+\phi_1(y,J z)J J x-2\phi_1(x,J y)J J z$,
\smallbreak\quad
$(\sigma_2\phi_2)(x,y)J z=\phi_2(x,J z)y-\phi_2(y,J z)x$
%\smallbreak\qquad
$-\phi_2(x,J J z)J y+\phi_2(y,J J z)J x$,
\smallbreak\quad
$(J\sigma_2\phi_2)(x,y)z=\phi_2(x,z)J y-\phi_2(y,z)J x$
%\smallbreak\qquad
$-\phi_2(x,J z)J J  y+\phi_2(y,J z)J J x$,
\smallbreak\quad
$(\sigma_3\phi_3)(x,y)J z=\phi_3(x,J z)y-\phi_3(y,J z)x$
\smallbreak\qquad $+2\phi_3(x,y)J z$
$-\phi_3(x,J J z)J y+\phi_3(y,J J z)J x$,
\smallbreak\quad
$(J\sigma_3\phi_3)(x,y)z=\phi_3(x,z)J y-\phi_3(y,z)J x$
\smallbreak\qquad
$+2\phi_3(x,y)J z$
$-\phi_3(x,J z)J J y+\phi_3(y,J z)J J  x$,
\smallbreak\quad
$(\sigma_4\phi_4)(x,y)J z=\phi_4(x,J z)y-\phi_4(y,J z)x$
\smallbreak\qquad$+2\phi_4(x,y)J z$
%\smallbreak\qquad
$-\phi_4(x,J J z)J y+\phi_4(y,J J z)J x$
%\smallbreak\qquad
$- 2\phi_4(x,J y)J J z$,
\smallbreak\quad
$(J\sigma_4\phi_4)(x,y)z=\phi_4(x,z)J y-\phi_4(y,z)J x$
\smallbreak\qquad$+2\phi_4(x,y)J z$
%\smallbreak\qquad
$-\phi_4(x,J z)J J y+\phi_4(y,J z)J J x$
%\smallbreak\qquad
$- 2\phi_4(x,J y)J J z$.
\medbreak Let $\{e_i\}$ be a basis for $V$ and let $\{e^i\}$ be the corresponding dual basis for $V^*$. We have
$e^i(J e_i)=\operatorname{Tr}(J)=0$. We examine the Ricci tensor:
\medbreak\quad
$(\rho\sigma_1\phi_1)(y,z)=\phi_1(e_i,z)e^i(y)-\phi_1(y,z)e^i(e_i)-\phi_1(e_i,J z)e^i(J y)$
\smallbreak\qquad
$+\phi_1(y,J z)e^i(J e_i)-2\phi_1(e_i,J y)e^i(J z)$
\smallbreak\qquad
$=\phi_1(y,z)-m\phi_1(y,z)-\phi_1(J y,J
z)+0-2\phi_1(J z,J y)$
%\smallbreak\qquad
$=-(m+2)\phi_1(y,z)$,
\smallbreak\quad
$(\rho\sigma_2\phi_2)(y,z)=\phi_2(e_i,z)e^i(y)-\phi_2(y,z)e^i(e_i)$
\smallbreak\qquad
$-\phi_2(e_i,J z)e^i(J y)+\phi_2(y,J z)e^i(J e_i)$
\smallbreak\qquad
$=\phi_2(y,z)-m\phi_2(y,z)-\phi_2(J y,J z)+0=(2-m)\phi_2(y,z)$,
\smallbreak\quad
$(\rho\sigma_3\phi_3)(y,z)=\phi_3(e_i,z)e^i(y)-\phi_3(y,z)e^i(e_i)$
\smallbreak\qquad
$+2\phi_3(e_i,y)e^i(z)-\phi_3(e_i,J z)e^i(J y)+\phi_3(y,J z)e^i(J e_i)$
\smallbreak\qquad
$=\phi_3(y,z)-m\phi_3(y,z)+2\phi_3(z,y)-\phi_3(J y,J z)+0=-(m+2)\phi_3(y,z)$,
\smallbreak\quad
$(\rho\sigma_4\phi_4)(y,z)=\phi_4(e_i,z)e^i(y)-\phi_4(y,z)e^i(e_i)$
\smallbreak\qquad
$+2\phi_4(e_i,y)e^i(z)-\phi_4(e_i,J z)e^i(J y)+\phi_4(y,J z)e^i(J e_i)- 2\phi_4(e_i,J y)e^i(J z)$
\smallbreak\qquad
$=\phi_4(y,z)-m\phi_4(y,z)+2\phi_4(z,y)
-\phi_4(J y,J z)$
\smallbreak\qquad
$+0-2\phi_4(J z,J y)={(-2-m)\phi_4(y,z)}$.
\medbreak\noindent The fact that $\sigma_i\phi_i$ takes values in the appropriate subspaces $\mathcal{K}_\star$ now follows from Theorem
\ref{thm-1.5}; it can also, of course, be checked directly.\end{proof}

\begin{remark}\label{rmk-2.3}
\rm Let $m\ge6$. We use Lemma \ref{lem-2.2} to split $\rho$ and see that there is a $\operatorname{GL}_{\mathbb{C}}^\star$ module decomposition of
$$\mathcal{K}_\pm\approx\Lambda_\pm^2\oplus S_\pm^2\oplus\{\mathcal{K}_\pm\cap\ker(\rho)\}\,.$$
By Theorem \ref{thm-1.5}, $\{\Lambda_+^2,\Lambda_-^2,S_+^2,S_-^2,\mathcal{K}_-,\mathcal{K}_+\}$ are inequivalent and non-trivial $\mathcal{U}^\star$
modules and hence, necessarily, inequivalent
$\operatorname{GL}_{\mathbb{C}}^\star$ modules as well. Theorem \ref{thm-1.5} also yields that
$\{\Lambda_+^2,\Lambda_-^2,S_+^2,S_-^2,\mathcal{K}_-\}$ are irreducible as $\mathcal{U}^\star$ modules and hence are irreducible as
$\operatorname{GL}_{\mathbb{C}}^\star$ modules as well. Thus to complete the proof of Theorem \ref{thm-1.2}, it suffices to show that
$\mathcal{K}_+\cap\ker(\rho)$ is an irreducible $\operatorname{GL}_{\mathbb{C}}^\star$ module; this will be done in Lemma \ref{lem-5.1} after first
establishing some preliminary algebraic results in Section~\ref{sect-3} and in Section \ref{sect-4}. The case $m=4$ is handled by
setting $\mathcal{K}_-\cap\ker(\rho)=W_{12}=\{0\}$ and deleting this module from the discussion below. \end{remark}

\section{The geometry of $\rho_{13}$}\label{sect-3}
If $\phi\in V^*\otimes V^*$, then set:
\medbreak\quad
$\vartheta(\phi)(x,y,z,w):= \phi (x,w)\langle y,z\rangle- \phi(y,w)\langle x,z\rangle$
\smallbreak\qquad
$+\phi(x,J w)\langle y,J z\rangle-\phi(y,J w) \langle x,J z\rangle- 2 \phi(z, J w)\langle x, J y\rangle$.
\begin{lemma}\label{lem-3.1}
Let $\phi\in S_{0,+}^2\oplus\Lambda_{0,+}^2$, let $\phi_1\in S_{0,+}^2$, and let $\phi_3\in \Lambda_{0,+}^2$.
\begin{enumerate}
\item $\vartheta\phi\in\mathcal{K}_+$.
\item  $\rho\sigma_1\phi_1=-(m+2)\phi_1$ and $\rho_{13}\sigma_1\phi_1=\phantom{-}2\phi_1$.
\item $\rho\sigma_3\phi_3=-(m+2)\phi_3$ and $\rho_{13}\sigma_3\phi_3=-2\phi_3$.
\item $\rho\vartheta\phi_1=\phantom{-}2\phi_1$ and $\rho_{13}\vartheta\phi_1=-(m+2)\phi_1$.
\item $\rho\vartheta\phi_3=-2\phi_3$ and $\rho_{13}\vartheta\phi_3=-(m+2)\phi_3$.
\end{enumerate}\end{lemma}
\begin{proof} It is immediate from the definition that $\vartheta(\phi)$ is anti-symmetric in the first 2 arguments. Note that
$$\phi(x,J y)=\phi(J x,J J y)=-\phi(J x,y)\,.
$$
We verify that
$\vartheta\phi$ satisfies the Bianchi identity by computing:
\medbreak\quad
$\vartheta(\phi)(x,y,z,w)+\vartheta(\phi)(y,z,x,w)+\vartheta(\phi)(z,x,y,w)$
\smallbreak\qquad
$= \phi (x,w)\langle y,z\rangle- \phi(y,w)\langle x,z\rangle$
\smallbreak\qquad
$\phantom{}+ \phi (y,w)\langle z,x\rangle- \phi(z,w)\langle y,x\rangle$
\smallbreak\qquad
$\phantom{}+ \phi (z,w)\langle x,y\rangle- \phi(x,w)\langle z,y\rangle$
\smallbreak\qquad
$\phantom{}+\phi(x,J w)\langle y,J z\rangle-\phi(y,J w) \langle x,J z\rangle- 2 \phi(z, J w)\langle x, J y\rangle$
\smallbreak\qquad
$\phantom{}+\phi(y,J w)\langle z,J x\rangle-\phi(z,J w) \langle y,J x\rangle- 2 \phi(x, J w)\langle y, J z\rangle$
\smallbreak\qquad
$\phantom{}+\phi(z,J w)\langle x,J y\rangle-\phi(x,J w) \langle z,J y\rangle- 2 \phi(y, J w)\langle z, J x\rangle$
$=0$.
\medbreak\noindent We will show that $\vartheta\phi\in\mathcal{K}_+$ by demonstrating that:
\medbreak\quad
$\vartheta\phi(x,y,z,w)=\vartheta\phi(x,y,J z,J w)=\vartheta\phi(J x,J y,z,w)$.
\medbreak\noindent We compare:
\medbreak\quad
$\vartheta(\phi)(x,y,z,w)= \phi (x,w)\langle y,z\rangle- \phi(y,w)\langle x,z\rangle$
\smallbreak\qquad
$+\phi(x,J w)\langle y,J z\rangle-\phi(y,J w) \langle x,J z\rangle- 2 \phi(z, J w)\langle x, J y\rangle$,
\medbreak\quad
$\vartheta(\phi)(x,y,J z,J w)= \phi (x,J w)\langle y,J z\rangle- \phi(y,J w)\langle x,J z\rangle$
\smallbreak\qquad
$+\phi(x,J J w)\langle y,J J z\rangle-\phi(y,J J w) \langle x,J J z\rangle- 2 \phi(J z, J J w)\langle x, J y\rangle$,
\medbreak\quad
$\vartheta(\phi)(J x,J y,z,w)= \phi (J x,w)\langle J y,z\rangle- \phi(J y,w)\langle J x,z\rangle$
\smallbreak\qquad
$+\phi(J x,J w)\langle J y,J z\rangle-\phi(J y,J w) \langle J x,J z\rangle- 2 \phi(z, J w)\langle J x, J J y\rangle$.
\medbreak\noindent We use Lemma~\ref{lem-2.2} to determine $\rho\sigma_1$ and $\rho\sigma_3$. We compute
$\rho\vartheta$:
\medbreak\quad
$\rho\vartheta(\phi)(y,z)=\varepsilon^{il}\phi (e_i,e_l)\langle y,z\rangle- \varepsilon^{il}\phi(y,e_l)\langle e_i,z\rangle$
\smallbreak\qquad
$+\varepsilon^{il}\phi(e_i,J e_l)\langle y,J z\rangle -\varepsilon^{il}\phi(y,J e_l) \langle e_i,J z\rangle- 2\varepsilon^{il} \phi(z, J e_l)\langle
e_i, J y\rangle$
\smallbreak\qquad
$=0-\phi(y,z)+0-\phi(y,J J z)-2\phi(z,J J y)$
\smallbreak\qquad
$=-\phi(y,z)+\phi(y,z)+2\phi(z,y)=2\phi(z,y)$.
\medbreak\noindent
We examine $\rho_{13}$. Let $\varepsilon_{ij}=\langle e_i,e_j\rangle$. Since $\phi\perp\langle\cdot,\cdot\rangle$ and since $\phi\perp\Omega$,
$\varepsilon^{il}\phi(e_i,e_l)=0$ and $\varepsilon^{il}\phi(e_i,Je_l)=0$.
\medbreak\quad
$(\rho_{13}\vartheta\phi)(y,w)= \varepsilon^{ik}\phi (e_i,w)\langle y,e_k\rangle-\varepsilon^{ik} \phi(y,w)\langle e_i,e_k\rangle$
\smallbreak\qquad
$+\varepsilon^{ik}\phi(e_i,J w)\langle y,J e_k\rangle-\varepsilon^{ik}\phi(y,J w) \langle e_i,J e_k\rangle
- 2 \varepsilon^{ik}\phi(e_k, J
w)\langle e_i, J y\rangle$
\smallbreak\qquad
$=\phi(y,w)-m\phi(y,w)-\phi(J y,J w)-0-2\phi(J y,J w)=-(m+2)\phi(y,w)$,
\medbreak\quad
$(\rho_{13}\sigma_1\phi_1)(y,w)=
\varepsilon^{ik}\phi_1(e_i,e_k)\langle y,w\rangle-\varepsilon^{ik}\phi_1(y,e_k)\langle e_i,w\rangle$
\smallbreak\qquad
$-\varepsilon^{ik}\phi_1(e_i,J e_k)\langle J y,w\rangle+\varepsilon^{ik}\phi_1(y,J e_k)\langle J e_i,w\rangle
-2\varepsilon^{ik}\phi_1(e_i,J y)\langle J e_k,w\rangle$
\smallbreak\qquad
$=0-\phi_1(y,w)-0-\phi_1(y,J J w)+2\phi_1(J w,J y)=2\phi_1(y,w)$,
\medbreak\quad
$ (\rho_{13}\sigma_3\phi_3)(y,w)=\varepsilon^{ik}\phi_3(e_i,e_k)\langle y,w\rangle
-\varepsilon^{ik}\phi_3(y,e_k)\langle e_i,w\rangle$
\smallbreak\qquad
$+2\varepsilon^{ik}\phi_3(e_i,y)\langle e_k,w\rangle-\varepsilon^{ik}\phi_3(e_i,J e_k)\langle J y,w\rangle
+\varepsilon^{ik}\phi_3(y,J e_k)\langle J e_i,w\rangle$
\smallbreak\qquad
$=0-\phi_3(y,w)+2\phi_3(w,y)-0-\phi_3(y,J J w)=-2\phi_3(y,w)$.\end{proof}
\medbreak\noindent

We use Theorem \ref{thm-1.5} to give a $\mathcal{U}^\star$ module decomposition into irreducible and inequivalent $\mathcal{U}^\star$ modules (where
as always we delete
$W_{11}$ if $m=4$):
$$
\mathcal{K}_+\cap\ker(\rho)=S_{0,+^2}\oplus \Lambda_{0,+}^2\oplus W_9\oplus W_{10}\oplus W_{11}\,.
$$
Let
$W_7$ (resp.
$W_8$) be the submodule of
$\mathcal{K}_+\cap\ker(\rho)$ which is isomorphic as a $\mathcal{U}^\star$ module to
$S_{0,+}^2$ (resp. $\Lambda_{0,+}^2$) under the map of $\rho_{13}$. Let $\pi_7$ (resp. $\pi_8$) be orthogonal projection on $W_7$ (resp. on $W_8$).
Let $\rho_{13,a}$ (resp. $\rho_{13,s}$) be the alternating (resp. symmetric) part of $\rho_{13}$.

\begin{lemma}\label{lem-3.2}
\ \begin{enumerate}
\item $\pi_7=-\frac{1}{m(m+4)}\{2\sigma_1+(m+2)\vartheta\}\rho_{13,s}$.
\smallbreak\item $\pi_8=-\frac{1}{m(m+4)}\{-2\sigma_3+(m+2)\vartheta\}\rho_{13,a}$.
\end{enumerate}\end{lemma}

\begin{proof}
We show that $\pi_7$ and $\pi_8$ split the action of $\rho_{13}$ on $\mathcal{K}_+\cap\ker(\rho)$ by using Lemma
\ref{lem-3.1} to see:
\medbreak\quad
$\rho\pi_7\phi_1=-\textstyle\frac1{m^2+4m}\{-(m+2)2+2(m+2)\}\phi_1=0$,
\smallbreak\quad
$\rho_{13}\pi_7\phi_1=-\textstyle\frac1{m^2+4m}\{4-(m+2)^2\}\phi_1=\phi_1$,
\medbreak\quad
$\rho\pi_8\phi_3=-\textstyle\frac1{m^2+4m}\{(m+2)2-2(m+2)\}\phi_3=0$,
\smallbreak\quad
$\rho_{13}\pi_8\phi_3=-\textstyle\frac1{m^2+4m}\{4-(m+2)^2\}\phi_3=\phi_3$.
\end{proof}

\section{The conjugate tensor}\label{sect-4}
Define the conjugate tensor $A^*$ by setting:
$$A^*(x,y,z,w):=A(x,y,z,J w)\,.$$
\begin{lemma}\label{lem-4.1}
The map $T:A\rightarrow A^*$ satisfies:
\begin{enumerate}
\item $T^2=-\id$.
\item $T$ is a $\operatorname{GL}_{\mathbb{C}}^\star$ module morphism intertwining the module $\mathcal{K}_+\cap\ker(\rho)$ with the module
$\{\mathcal{K}_+\cap\ker(\rho)\}\otimes\chi$.
\smallbreak
\item $T$ is a $\mathcal{U}^\star$ module morphism which intertwines $W_9$ with $W_{10}\otimes\chi$, which intertwines
$W_7$ with $W_8\chi$, and which intertwines $W_{11}$ with $W_{11}\otimes\chi$.
\end{enumerate}
\end{lemma}

\begin{remark}\rm Since $J$ appears an odd number of times in the definition of $T$, it is necessary to introduce the $\mathbb{Z}_2$ valued
representation $\chi$ to take this into account. Since $\chi^2$ is the trivial representation, this result also yields that $T$ intertwines $W_{10}$
with $W_9\otimes \chi$ and that $T$ intertwines $W_8$ with $W_7\otimes\chi$.
\end{remark}

\begin{proof} Assertion (1) is immediate. Let $A\in\mathcal{K}_+\cap\ker(\rho)$. By expressing
\begin{eqnarray*}
A^*(x,y,z,w)&=&A(x,y,z,J w)=A(x,y,J z,J J w)\\
&=&-A(x,y,J z,w),
\end{eqnarray*}
we see that $\rho(A^*)(y,z)=-\rho(A)(y,J z)$ and thus $T$ preserves $\ker(\rho)$. It is immediate that $A^*$ satisfies the Bianchi
identity\index{Bianchi identity} and that
\medbreak\quad
$A^*(x,y,J z,J w)=A(x,y,J z,J J w)= A(x,y,z,J w)=A^*(x,y,z,w),$
\smallbreak\quad
$A^*(J x,J y,z,w)=A(J x,J y,z,J w)= A(x,y,z,J w)=A^*(x,y,z,w)$.
\medbreak\noindent
Assertion (2) now follows.
If $\psi\in\otimes^2V^*$, we define $T\psi(x,y):=\psi(x,J y)$. It is then immediate that
$\rho_{13}TA=T\rho_{13}A$. Consequently $T$ preserves the subspace $\ker(\rho_{13})\cap\mathcal{K}_+=W_9\oplus W_{10}\oplus W_{11}$. We have:
$$
\phi\in S_{0,+}^2 \quad\Rightarrow\quad T\phi\in\Lambda_{0,+}^2 \quad\text{and}
\quad
\psi\in \Lambda_{0,+}^2 \quad\Rightarrow\quad T\psi\in S_{0,+}^2.
$$
We see that $T$ intertwines the representation $W_9$ with $W_{10}\otimes\chi$ by applying these relations to the last indices of a $4$-tensor. Since
$T$ is an isometry, $T$ intertwines $W_{11}$ with $W_{11}\otimes\chi$ since $W_{11}$ is the orthogonal complement of $W_9\oplus
W_{10}$ in the module $\mathcal{K}_+\cap\ker(\rho)\cap\ker(\rho_{13})$. Since $W_7\oplus W_8$ is the orthogonal complement of $W_9\oplus W_{10}\oplus
W_{11}$ in
$\mathcal{K}_+\cap\ker(\rho)$, $T$ preserves the subspace $W_7\oplus W_8$. Since $T$ interchanges $S_{0,+}^2$ and $\Lambda_{0,+}^2$ and since $T$
commutes with
$\rho_{13}$, $T$ interchanges the subspaces $W_7$ and $W_8$ and consequently intertwines the representation $W_7$ with $W_8\otimes\chi$.\end{proof}

\begin{lemma}\label{lem-4.3}
Let $\pi_i$ for $i=9,10,11$ be orthogonal projection on the $\mathcal{U}^\star$ modules $W_i$.
Let $A\in\mathcal{K}_+\cap\ker(\rho)$.
\begin{enumerate}
\item If $\rho_{13}(A)\in\Lambda_{0,+}^2$, then
$\pi_9(A)(x,y,z,w)$
\smallbreak\qquad
$=\frac{1}{4}\{A(x,y,z,w)+A(y,x,w,z)+A(z,w,x,y)+A(w,z,y,x)\}$.
\smallbreak\item If $\rho_{13}(A)\in S_{0,+}^2$, then
$\pi_{10}(A)(x,y,z,w)
=-\frac{1}{4}\{A(x,y,z,J J w)$
\smallbreak\qquad$+A(y,x,J w,J z)+A(z,J w,x,J y)+A(J w,z,y,J x)\}$.
\smallbreak\item
$\pi_{11}(A)=\id-\pi_9-\pi_{10}$.
\end{enumerate}
\end{lemma}

\begin{proof} Clearly $\pi_9(A)$ is anti-symmetric in $(x,y)$. We verify that $\pi_9(A)$ satisfies
the Bianchi identity and show $\pi_9(A)\in\mathfrak{A}$ by computing:
\medbreak\quad
$\pi_9(A)(x,y,z,w)+\pi_9(A)(y,z,x,w)+\pi_9(A)(z,x,y,w)$
\smallbreak\qquad
$=\frac14\{A(x,y,z,w)+A(y,x,w,z)+A(z,w,x,y)+A( w,z,y,x)\}$
\smallbreak\qquad
$\phantom{}+\frac14\{A(y,z,x,w)+A(z,y,w,x)+A(x,w,y,z)+A( w,x,z,y)\}$
\smallbreak\qquad
$\phantom{}+\frac14\{A(z,x,y,w)+A(x,z,w,y)+A(y,w,z,x)+A( w,y,x,z)\}$
\smallbreak\qquad
$=\frac14\{A( w,z,y,x)+A(z,y,w,x)+A(y,w,z,x)\}$
\smallbreak\qquad
$\phantom{}+\frac14\{A(z,w,x,y)+A( w,x,z,y)+A(x,z,w,y)\}$
\smallbreak\qquad
$\phantom{}+\frac14\{A(y,x,w,z)+A(x,w,y,z)+A( w,y,x,z)\}$
\smallbreak\qquad
$\phantom{}+\frac14\{A(x,y,z,w)+A(y,z,x,w)+A(z,x,y,w)\}=0$.
\smallbreak\noindent We show $\pi_9(A)\in\mathcal{K}_+$ by comparing:
\medbreak\quad
$\pi_9(A)(x,y,z,w)$
\smallbreak\qquad
$=\frac{1}{4}\{A(x,y,z,w)+A(y,x,w,z)+A(z,w,x,y)+A(w,z,y,x)\}$,
\medbreak\quad
$\pi_9(A)(x,y,J z,J w)$
\smallbreak\qquad
$=\frac{1}{4}\{A(x,y,J z,J w)+A(y,x,J w,J z)+A(J z,J w,x,y)+A(J w,J z,y,x)\}$,
\medbreak\quad
$\pi_9(A)(J x,J y,z,w)$
\smallbreak\qquad
$=\frac{1}{4}\{A(J x,J y,z,w)+A(J y,J x,w,z)+A(z,w,J x,J y)+A(w,z,J y,J x)\}$.
\medbreak\noindent As $\pi_9(A)$ is anti-symmetric in the last two indices,
$\rho(\pi_9(A))=-\rho_{13}(\pi_9(A))$. We assume that $\rho(A)=0$ and that $\rho_{13}(A)$ is anti-symmetric. We show
$\pi_9(A)\in\ker(\rho)$ and therefore that
$\pi_9(A)$ takes values in $W_{9}$ by computing:
\medbreak\quad
$\rho(\pi_9(A))(y,z)=\frac14\varepsilon^{il}A(e_i,y,z,e_l)+\frac14\varepsilon^{il}A(y,e_i,e_l,z)$
\smallbreak\qquad
$+\frac14\varepsilon^{il}A(z,e_l,e_i,y)+\frac14\varepsilon^{il}A(e_l,z,y,e_i)$
\smallbreak\qquad
$=\frac14\{\rho(y,z)-\rho_{13}(y,z)-\rho_{13}(z,y)+\rho(z,y)\}=0$.
\medbreak\noindent Suppose $A$ is anti-symmetric in $(z,w)$. Then it is easily checked that $A\in\mathfrak{R}$ and hence
$\pi_9(A)(x,y,z,w)=A(x,y,z,w)$. This completes the proof of Assertion (1).

By Lemma \ref{lem-4.1}, $T$ maps the subspace $W_9$ to the subspace $W_{10}$; the factor of $\chi$ is only added to take into account the
equivariance and plays no role in the analysis. Since $T^{-1}=-T$ and since
$T$ is an isometry, we have therefore that
$- T\pi_9T=\pi_{10}$; Assertion (2) now follows from Assertion (1); $T\rho_{13}=\rho_{13}T$ and $T$ interchanges the subspaces $\Lambda_{0,+}^2$
with $S_{0,+}^2$. Assertion (3) is
immediate from Assertions (1) and (2) and from Theorem~\ref{thm-1.5}.
\end{proof}

\section{The proof of Theorem \ref{thm-1.2}}\label{sect-5}
As noted in Remark \ref{rmk-2.3}, we may complete the proof of Theorem~\ref{thm-1.2},
by showing:
\begin{lemma}\label{lem-5.1}
$\ker(\rho)\cap\mathcal{K}_+$ is an irreducible $\operatorname{GL}_{\mathbb{C}}^\star$ module.
\end{lemma}

\begin{proof} We suppose to the contrary that $\xi$ is a non-trivial proper $\operatorname{GL}_{\mathbb{C}}^\star$
submodule of
$\ker(\rho)\cap\mathcal{K}_-$. We introduce an auxiliary Hermitian
inner product $\langle\cdot,\cdot\rangle$. We apply Theorem \ref{thm-1.5}. The modules $\{W_7,W_8,W_9,W_{10},W_{11}\}$ are inequivalent and
irreducible $\mathcal{U}^\star$ modules (we delete $W_{11}$ from consideration if $m=4$). Thus there is a set of indices
$I\subset\{7,8,9,10,11\}$ so:
$$\xi=\oplus_{i\in I}W_i\,.$$
We choose an orthonormal basis $\{e_1,f_1,...,e_{\bar m},f_{\bar m}\}$ for $V$ so $J e_i=f_i$ and $J
f_i=- e_i$. All $4$-tensors considered in the proof of Lemma~\ref{lem-5.1} will be anti-symmetric in the first $2$ indices.
\subsection{Suppose that $W_{9}\subset\xi$}\label{sect-5.1}. Let $A$ be determined by the relations:
\medbreak\qquad
$A(e_1,f_1,e_1,f_2)=-1$, $A(e_1,f_1,f_1,e_2)=\phantom{-}1$,
\smallbreak\qquad
$A(e_1,f_1,e_2,f_1)=-1$, $A(e_1,f_1,f_2,e_1)=\phantom{-}1$,
\smallbreak\qquad
$A(e_1,f_2,e_1,f_1)=-1$, $A(e_1,f_2,f_1,e_1)=\phantom{-}1$,
\smallbreak\qquad
$A(e_1,f_2,e_2,f_2)=\phantom{-}1$, $A(e_1,f_2,f_2,e_2)=-1$,
\smallbreak\qquad
$A(f_1,e_2,e_1,f_1)=\phantom{-}1$, $A(f_1,e_2,f_1,e_1)=-1$,
\smallbreak\qquad
$A(f_1,e_2,e_2,f_2)=-1$, $A(f_1,e_2,f_2,e_2)=\phantom{-}1$,
\smallbreak\qquad
$A(e_2,f_2,e_1,f_2)=\phantom{-}1$, $A(e_2,f_2,f_1,e_2)=-1$,
\smallbreak\qquad
$ A(e_2,f_2,e_2,f_1)=\phantom{-}1$, $A(e_2,f_2,f_2,e_1)=-1$.
\medbreak\noindent
It is then immediate by inspection that $A\in W_{9}$. Let
$$
\begin{array}{l}
g_{1,\varepsilon}(e_i):=\left\{\begin{array}{lll}\varepsilon
e_1&\text{if}&i=1\\e_i&\text{if}&i\ne1\end{array}\right\},
\quad
g_{1,\varepsilon}(e^i):=\left\{\begin{array}{lll}\varepsilon^{-1}e^1&\text{if}&i=1\\
e^i&\text{if}&i\ne1\end{array}\right\},\\
g_{1,\varepsilon}(f_i):=\left\{\begin{array}{lll}\varepsilon f_1&\text{if}&i=1\\
f_i&\text{if}&i\ne1\end{array}\right\},\quad
   g_{1,\varepsilon}(f^i):=\left\{\begin{array}{lll}\varepsilon^{-1}f^1&\text{if}&i=1\\
f^i&\text{if}&i\ne1\end{array}\right\}.
\end{array}$$
Since $\xi$ is a finite dimensional linear subspace, it is closed. Consequently
$$
B_1:=\lim_{\varepsilon\rightarrow0}\varepsilon g_{1,\varepsilon}^*A\in\xi\,.
$$
The non-zero
components of $B_1$ and $\rho_{13}$ are determined by:
$$\begin{array}{rr}
B_1(e_2,f_2,e_2,f_1)=1,&\qquad B_1(e_2,f_2,f_2,e_1)=-1,\\
\rho_{13}(B_1)(e_2,e_1)=1,&\rho_{13}(B_1)(f_2,f_1)=1\,.\gronko
\end{array}$$
By interchanging the roles of $\{e_1,f_1\}$ and $\{e_2,f_2\}$ we can create an element $B_2\in\xi$ with
$$\begin{array}{rr}
B_2(e_1,f_1,e_1,f_2)=1,&\qquad B_2(e_1,f_1,f_1,e_2)=-1,\\
\rho_{13}(B_2)(e_1,e_2)=1,&\rho_{13}(B_2)(f_1,f_2)=1.\gronko
\end{array}$$
Thus $B_1+B_2$ has a non-zero component in $W_7$ and $B_1-B_2$ has a non-zero component in $W_8$. This shows
that:
$$W_9\subset\xi\quad\Rightarrow\quad W_7\oplus W_8\subset\xi\,.$$
Let $B_i^*:=TB_i$. We study $\pi_{10}(B_1+B_2)$ by examining
$\pi_9(B_1^*+B_2^*)$:
$$\begin{array}{rr}
(B_1^*+B_2^*)(e_1,f_1,e_1,e_2)=1,&\quad(B_1^*+B_2^*)(e_1,f_1,f_1,f_2)=1,\\
(B_1^*+B_2^*)(e_2,f_2,e_2,e_1)=1,&\quad(B_1^*+B_2^*)(e_2,f_2,f_2,f_1)=1,\gronko\\
\rho_{13}(B_1^*+B_2^*)(f_1,e_2)=1,&\rho_{13}(B_1^*+B_2^*)(e_2,f_1)=-1,\gronko\\
\rho_{13}(B_1^*+B_2^*)(f_2,e_1)=1,&\rho_{13}(B_1^*+B_2^*)(e_1,f_2)=-1.\gronko
\end{array}$$
Since $\rho_{13}(B_1^*+B_2^*)$ is anti-symmetric, we have by Lemma~\ref{lem-4.3} that:
$$\pi_9(B_1^*+B_2^*)(e_1,f_1,e_1,e_2)=\textstyle\frac14\,.$$
Consequently $\pi_{10}(B_1+B_2)\ne0$. This implies:
$$W_9\subset\xi\quad\Rightarrow\quad W_{10}\subset\xi\,.$$
Suppose $m\ge6$. Set
$$\begin{array}{l}
g_{2,\varepsilon}(e_i):=\left\{\begin{array}{lrr}
e_3-\varepsilon e_1&\text{ if}&i=3\\e_i&\text{ if}&i\ne3\end{array}\right\},
\quad
g_{2,\varepsilon}(e^i):=\left\{\begin{array}{lrr}e^1+\varepsilon e^3&\text{ if}&i=1\\
e^i&\text{if}&i\ne1\end{array}\right\},\\
g_{2,\varepsilon}(f_i):=\left\{\begin{array}{lrr}f_3-\varepsilon f_1&\text{ if}&i=1\\
f_i&\text{if}&i\ne3\end{array}\right\},\quad
   g_{2,\varepsilon}(f^i):=\left\{\begin{array}{lrr}f^1+\varepsilon f^3&\text{if}&i=1\\
f^i&\text{if}&i\ne1\end{array}\right\}.
\end{array}$$
Let $B_3:=\partial_\varepsilon\{g_{2,\varepsilon}^*A\}|_{\varepsilon=0}$. We then have:
\medbreak\qquad
$B_3(e_1,f_1,e_2,f_3)=-1$, $B_3(e_1,f_1,f_2,e_3)=\phantom{-}1$,
\smallbreak\qquad
$B_3(e_1,f_2,e_1,f_3)=-1$, $B_3(e_1,f_2,f_1,e_3)=\phantom{-}1$,
\smallbreak\qquad
$B_3(f_1,e_2,e_1,f_3)=\phantom{-}1$, $B_3(f_1,e_2,f_1,e_3)=-1$,
\smallbreak\qquad
$ B_3(e_2,f_2,e_2,f_3)=\phantom{-}1$, $B_3(e_2,f_2,f_2,e_3)=-1$.
\medbreak\noindent
We use Lemma \ref{lem-4.3} to see $|\pi_9B_3(e_1,f_1,e_2,e_3)|\le\frac14$ and
$|\pi_{10}(e_1,f_1,e_2,e_3)|\le\frac14$. Since $B_3\in\ker(\rho_{13})$, we have $|\pi_{11}B_3(e_1,f_1,e_2,e_3)|\ge\frac12$ and thus
$W_{11}\subset\xi$. We summarize our conclusions:
$$W_9\subset\xi\quad\Rightarrow\quad\xi=\mathcal{K}_+\cap\ker(\rho)\,.$$

\subsection{Suppose that $W_7\subset\xi$}\label{sect-5.2}. We clear the previous notation. Let
$$\phi:=e^1\otimes e^2+e^2\otimes e^1
   + f^1\otimes f^2+ f^2\otimes f^1\in S_{0,+}^2\,.
$$
We use Lemma~\ref{lem-3.2} to find $A\in W_7$ so that $\rho_{13}A=\phi$. We shall not compute all the
terms in $A$ as this would be a bit of a bother and shall content ourselves with determining just a few terms.
We compute:
\medbreak\quad
$\langle\sigma_1\phi(e_2,f_2)e_2,e_1\rangle=0$,\qquad
$\langle\sigma_1\phi(e_2,f_2)e_2,f_1\rangle=0$,
\smallbreak\quad
$\vartheta(\phi)(e_2,f_2,e_2,e_1):= \phi (e_2,e_1)\langle f_2,e_2\rangle- \phi(f_2,e_1)\langle e_2,e_2\rangle$
\smallbreak\qquad
$+\phi(e_2,J e_1)\langle f_2,J e_2\rangle-\phi(f_2,J e_1) \langle e_2,J e_2\rangle
- 2 \phi(e_2, J e_1)\langle e_2, J f_2\rangle=0$,
\smallbreak\quad
$\vartheta(\phi)(e_2,f_2,e_2,f_1):= \phi (e_2,f_1)\langle f_2,e_2\rangle- \phi(f_2,f_1)\langle e_2,e_2\rangle$
\smallbreak\qquad
$+\phi(e_2,J f_1)\langle f_2,J e_2\rangle-\phi(f_2,J f_1) \langle e_2,J e_2\rangle$
\smallbreak\qquad
$- 2 \phi(e_2, J f_1)\langle e_2, J f_2\rangle=0-1-1-0-2\ne0$.
\medbreak\quad
$0=c_1:=A(e_2,f_2,e_2,e_1)$,\qquad$0\ne c_2:=A(e_2,f_2,e_2,f_1)$.
\medbreak\noindent
Let $\Phi\in\mathcal{U}$ be defined by:
$$\Phi e_i:=\left\{\begin{array}{rll}
-e_1&\text{if}&i=1\\e_i&\text{if}&i>1\end{array}\right\},\qquad
\Phi f_i:=\left\{\begin{array}{rll}
-f_1&\text{if}&i=1\\f_i&\text{if}&i>1\end{array}\right\}\,.
$$
Since $\Phi^*\phi=-\phi$, we have $\Phi^*A=-A$.
Thus the number of times that $x_i$ is $e_1$ or $f_1$ is odd; similarly, the number of times that $x_i$ is $f_1$ or $f_2$ is odd as well.
Define
$g_{\varepsilon_1,\varepsilon_2}\in\operatorname{GL}_{\mathbb{C}}^\star$ by setting:
\begin{eqnarray*}
&&g_{\varepsilon_1,\varepsilon_2}e_i=\left\{\begin{array}{lll}
\varepsilon_1e_1&\text{if}&i=1\\
\varepsilon_2e_2&\text{if}&i=2\\
e_i&\text{if}&i\ge3\end{array}\right\},\quad
g_{\varepsilon_1,\varepsilon_2}e^i=\left\{\begin{array}{lll}
\varepsilon_1^{-1}e^1&\text{if}&i=1\\
\varepsilon_2^{-1}e^2&\text{if}&i=2\\
e^i&\text{if}&i\ge3\end{array}\right\},\\
&&g_{\varepsilon_1,\varepsilon_2}f_i=\left\{\begin{array}{lll}
\varepsilon_1f_1&\text{if}&i=1\\
\varepsilon_2f_2&\text{if}&i=2\\
f_i&\text{if}&i\ge3\end{array}\right\},\quad
g_{\varepsilon_1,\varepsilon_2}f^i=\left\{\begin{array}{lll}
\varepsilon_1^{-1}f^1&\text{if}&i=1\\
\varepsilon_2^{-1}f^2&\text{if}&i=2\\
f^i&\text{if}&i\ge3\end{array}\right\}.
\end{eqnarray*}
Expand $g_{\varepsilon_1,\varepsilon_2}^*A$ as a finite Laurent polynomial in $\{\varepsilon_1,\varepsilon_2\}$.
As $g_{\varepsilon_1,\varepsilon_2}^*A\in\xi$, all the coefficient curvature tensors also belong to $\xi$.
Let $B\in\xi$ be the coefficient of $\varepsilon_1^{-1}\varepsilon_2^3$ in $g_{\varepsilon_1,\varepsilon_2}^*A$;
$$
B=\textstyle\left.\left\{\frac16\varepsilon_1\partial_{\varepsilon_2}^3
g_{\varepsilon_1,\varepsilon_2}^*A\right\}\right|_{\varepsilon_1=0,\varepsilon_2=0}\,.
$$
The
only (possibly) non-zero components of $B$ are given by:
\medbreak
$B(e_2,f_2,e_2,e_1)=A(e_2,f_2,e_2,e_1)=0$,
\smallbreak
$B(e_2,f_2,e_2,f_1)=A(e_2,f_2,e_2,f_1)=c_2$,
\smallbreak
$B(e_2,f_2,f_2,e_1)=-B(e_2,f_2,e_2,f_1)=-c_2$,
\smallbreak
$B(e_2,f_2,f_2,f_1)= B(e_2,f_2,e_2,e_1)=0$.
\medbreak\noindent We examine:
$$
\rho_{13}(B)(e_2,e_1)= c_2\quad\text{and}\quad\rho_{13}(B)(f_2,f_1)=c_{2}\,.
$$
Interchanging the roles of the indices ``1" and ``2" is an isometry which preserves
$\phi$; this creates a tensor
$\tilde B\in\xi$ so that
$$\begin{array}{ll}
\tilde B(e_1,f_1,f_1,e_2)=-c_2,&\tilde B(e_1,f_1,e_1,f_2)=c_2,\\
\rho_{13}(\tilde B)(e_1,e_2)= c_2,\quad&\rho_{13}(\tilde B)(f_1,f_2)=c_{2}.\gronko
\end{array}$$
In particular $B-\tilde B$ has an anti-symmetric Ricci tensor so we may use Lemma~\ref{lem-4.3}
to compute
$$\pi_9(B-\tilde B)(e_2,f_2,e_2,f_1)=\textstyle\frac14c_2\ne0\,.$$
This implies $W_{9}\subset\xi$ and hence by Section~\ref{sect-5.1},
$$W_7\subset\xi\quad\Rightarrow\quad W_9\subset\xi\quad\Rightarrow\quad\xi=\mathcal{K}_+\cap\ker(\rho)\,.$$

\subsection{Suppose that $m\ge6$ and that $W_{11}\subset\xi$}.
Clear the previous notation. Set:
\medbreak\qquad
$A(e_1,e_2,e_1,e_3)=\phantom{-}{1}$, $A(e_1,e_2,f_1,f_3)=\phantom{-}{1}$,
\smallbreak\qquad
$A(e_1,f_2,e_1,f_3)={-1}$, $A(e_1,f_2,f_1,e_3)={\phantom{-}1}$,
\smallbreak\qquad
$A(e_1,e_3,e_1,e_2)={-1}$, $A(e_1,e_3,f_1,f_2)={-1}$,
\smallbreak\qquad
$A(e_1,f_3,e_1,f_2)={\phantom{-}1}$, $A(e_1,f_3,f_1,e_2)={-1}$,
\smallbreak\qquad
$A(f_1,e_2,e_1,f_3)=\phantom{-}{1}$, $A(f_1,e_2,f_1,e_3)={-1}$,
\smallbreak\qquad
$A(f_1,f_2,e_1,e_3)={\phantom{-}1}$, $A(f_1,f_2,f_1,f_3)=\phantom{-}{1}$,
\smallbreak\qquad
$A(f_1,e_3,e_1,f_2)={-1}$, $A(f_1,e_3,f_1,e_2)={\phantom{-}1}$,
\smallbreak\qquad
$A(f_1,f_3,e_1,e_2)={-1}$, $A(f_1,f_3,f_1,f_2)={-1}$\,.
\medbreak\noindent
We verify by inspection that $A\in\mathcal{K}_+\cap\ker(\rho)\cap\ker(\rho_{13})$.
We study:
\begin{eqnarray*}
\pi_9(A)(x,y,z,w)&=&\textstyle\frac{1}{4}\{A(x,y,z,w)+A(y,x,w,z)\\
&&\quad+A(z,w,x,y)+A(w,z,y,x)\}\,.
\end{eqnarray*}
Let $U_1$ denote the set of elements $\{e_2,f_2,e_3,f_3\}$. For $\pi_9(A)$ to be non-zero, either $x\in U_1$ or $y\in U_1$ and either $z\in U_1$ or
$w\in U_1$. If
$x$ and
$z$ belong to
$U_1$, then we have that $A(x,y,z,w)=-A(z,w,x,y)$ and that $A(y,x,w,z)=A(w,z,y,x)=0$. Thus
$\pi_9 A(x,y,z,w)=0$ in this special case. Since $\pi_9 A$ is
anti-symmetric in the first 2 indices and in the last 2 indices, we see that $\pi_9 A=0$ in
the remaining cases. To examine $\pi_{10}$, we consider the dual tensor $A^*=TA$:
\medbreak\qquad
$A^*(e_1,e_2,e_1,f_3)={-1}$, $A^*(e_1,e_2,f_1,e_3)={\phantom{-}1}$,
\smallbreak\qquad
$A^*(e_1,f_2,e_1,e_3)={-1}$, $A^*(e_1,f_2,f_1,f_3)={-1}$,
\smallbreak\qquad
$A^*(e_1,e_3,e_1,f_2)={\phantom{-}1}$, $A^*(e_1,e_3,f_1,e_2)={-1}$,
\smallbreak\qquad
$A^*(e_1,f_3,e_1,e_2)={\phantom{-}1}$, $A^*(e_1,f_3,f_1,f_2)=\phantom{-}{1}$,
\smallbreak\qquad
$A^*(f_1,e_2,e_1,e_3)={\phantom{-}1}$, $A^*(f_1,e_2,f_1,f_3)=\phantom{-}{1}$,
\smallbreak\qquad
$A^*(f_1,f_2,e_1,f_3)={-1}$, $A^*(f_1,f_2,f_1,e_3)=\phantom{-}{1}$,
\smallbreak\qquad
$A^*(f_1,e_3,e_1,e_2)={-1}$, $A^*(f_1,e_3,f_1,f_2)={-1}$,
\smallbreak\qquad
$A^*(f_1,f_3,e_1,f_2)=\phantom{-}{1}$, $A^*(f_1,f_3,f_1,e_2)={-1}$\,.
\medbreak\noindent
Once again $x\in U_1$ and $z\in U_1$ implies $A^*(x,y,z,w)+A^*(z,w,x,y)=0$  while
$A^*(y,x,w,z)=A^*(w,z,y,x)=0$. The argument given above to show that $\pi_9 A=0$ then shows
$\pi_9 A^*=0$ and hence $\pi_{10} A=0$. Consequently since
$\rho(A)=\rho_{13}(A)=0$, we may conclude that $A\in W_{11}$.
Set:
\begin{eqnarray*}
&&g_\varepsilon(e_i):=\left\{\begin{array}{lll}\varepsilon
e_3&\text{if}&i=3\\e_i&\text{if}&i\ne3\end{array}\right\},\quad
g_\varepsilon(e^i):=\left\{\begin{array}{lll}\varepsilon^{-1}e^3&\text{if}&i=3\\e^i&\text{if}&i\ne3\end{array}\right\},\\
&&g_\varepsilon(f_i):=\left\{\begin{array}{lll}\varepsilon
f_3&\text{if}&i=3\\f_i&\text{if}&i\ne3\end{array}\right\},\quad
 g_\varepsilon(f^i):=\left\{\begin{array}{lll}\varepsilon^{-1}f^3
&\text{if}&i=3\\f^i&\text{if}&i\ne3\end{array}\right\}.
\end{eqnarray*}
We set $B:=\lim_{\varepsilon\rightarrow0}\varepsilon g_\varepsilon^*A\in\xi$. We see that the non-zero components of
$B$ are determined by:
\medbreak\qquad
$B(e_1,e_2,e_1,e_3)=\phantom{-}1$, $B(e_1,e_2,f_1,f_3)=\phantom{-}{1}$,
\smallbreak\qquad
$B(e_1,f_2,e_1,f_3)=-{1}$, $B(e_1,f_2,f_1,e_3)=\phantom{-}{1}$,
\smallbreak\qquad
$B(f_1,e_2,e_1,f_3)=\phantom{-}{1}$, $B(f_1,e_2,f_1,e_3)=-{1}$,
\smallbreak\qquad
$B(f_1,f_2,e_1,e_3)=\phantom{-}{1}$, $B(f_1,f_2,f_1,f_3)=\phantom{-}1$.
\medbreak\noindent
We verify that $\rho(B)=\rho_{13}(B)=0$. We use Lemma~\ref{lem-4.3} to see:
\begin{eqnarray*}
&&\pi_9(B)(e_1,e_2,e_1,e_3)=\textstyle\frac14B(e_1,e_2,e_1,e_3)=\textstyle\frac14,\\
&&\pi_{10}(B)(e_1,e_2,e_1,e_3)
=-\textstyle\frac14\pi_9(B^*)(e_1,e_2,e_1,f_3)\\
&&\qquad=\textstyle\frac14B(e_1,e_2,e_1,e_3)=\textstyle\frac14\,.
\end{eqnarray*}
We use Section~\ref{sect-5.1} to see that if $m\ge6$, then
$$
W_{11}\subset\xi\quad\Rightarrow\quad
  W_9\subset\xi\quad\Rightarrow\quad\xi=\mathcal{K}_+\cap\ker(\rho)\,.
$$

\subsection{Suppose that $W_{10}\subset\xi$}. We use Lemma~\ref{lem-4.1} to
interchange the roles of $W_9$ and $W_{10}$ and then apply the results of Section \ref{sect-5.1} to see:
\begin{eqnarray*}
&&W_{10}\subset\xi\quad\Rightarrow\quad W_9\subset
T\xi\quad\Rightarrow\quad T\xi=\mathcal{K}_+\cap\ker(\rho)\quad\Rightarrow\quad
\xi=\mathcal{K}_+\cap\ker(\rho)\,.
\end{eqnarray*}

\subsection{Suppose that $W_8\subset\xi$}. We use the duality operator and Section \ref{sect-5.2} to see:
\begin{eqnarray*}
&&W_8\subset\xi\quad\Rightarrow\quad W_7\subset
T\xi\quad\Rightarrow\quad T\xi=\mathcal{K}_+\cap\ker(\rho)\quad\Rightarrow\quad
\xi=\mathcal{K}_+\cap\ker(\rho)\,.
\end{eqnarray*}This completes the proof of Lemma~\ref{lem-5.1} and thereby of all the assertions in this paper.\end{proof}

\section*{Acknowledgments}
Research of all authors was supported by project MTM2009-07756 (Spain). The
research M. Brozos-V\'azquez was also supported by
INCITE09 207 151 PR (Spain) and the research of S. Nik\v cevi\'c was also supported by project 144032 (Serbia).


\begin{thebibliography}{aaa}

\bibitem{BNGS06}  N. Bla\v zi\'c, P. Gilkey, S. Nik\v cevi\'c, and U. Simon,
``Algebraic theory of affine curvature tensors",
{\it Archivum Mathematicum}, Masaryk University (Brno, Czech Republic) ISSN 0044-8753,
tomus 42 (2006), supplement: Proceedings of the
26th Winter School of Geometry and Physics 2006 (SRNI), 147--168.


\bibitem{B90} N. Bokan,
``On the complete decomposition of curvature tensors of Riemannian  manifolds with symmetric connection'',
 {\it Rend. Circ. Mat. Palermo}  {\bf XXIX} (1990), 331--380.

\bibitem{BGN10a} M. Brozos-V\'azquez, P. Gilkey, and  S. Nik\v cevi\'c,
``Geometric realizations of affine K\"ahler curvature models'',
 {\it Results. Math.} {\bf 59} (2011), 507--521.

\bibitem{BGGH11} M. Brozos-V\'azquez,  E. Garc\'{\i}a-R\'{\i}o,  P. Gilkey, and L. Hervella,
``Geometric realizability of covariant derivative K\"ahler tensors
for almost pseudo-Hermitian and almost para-Hermitian manifolds", to appear {\it Ann. Mat. Pura Appl.} (2011).

\bibitem{Ca10} A. Caldarella, ``On paraquaternionic submersions between paraquaternionic K\"ahler manifolds",
{\it Acta Appl. Math.} {\bf 112} (2010), 1--14.

\bibitem{C11} C. Coevering, ``Examples of asymptotically conical Ricci-flat K\"ahler manifolds", {\it Math. Z.} {\bf 267} (2011), 465--496.

\bibitem{S93} V. De Smedt,
``Decomposition of the curvature tensor of Hyper-Kaehler manifolds",
{\it Letters in Math. Physics} {\bf 30} (1994), 105--117.

\bibitem{GV} E. Garc\'{\i}a-R\'{\i}o and L. Vanhecke, ``Five-dimensional $\varphi$-symmetric spaces", {\it Balkan J. Geom. Appl.} {\bf 1}
(1996), 31--44.

\bibitem{GNU09} P. Gilkey, S. Nik\v cevi\'c, and U. Simon,
``Geometric theory of equiaffine curvature tensors'',
{\it Result. Math.} {\bf 56} (2009), 275--317.


\bibitem{GNS09} P. Gilkey, S. Nik\v cevi\'c, and U. Simon,
``Geometric realizations, curvature decompositions, and Weyl manifolds",
{\it J. Geom. Phys.} {\bf 61} (2011), 270--275.

\bibitem{GL11} M. G\"oteman and U. Lindstr\"om, ``Pseudo-hyperk\"ahler geometry and generalized K\"ahler geometry",
{\it Lett. Math. Phys.} {\bf 95} (2011), 211-222.

\bibitem{gray} A. Gray,
``Curvature identities for Hermitian and almost Hermitian manifolds",
{\it T{\^o}hoku Math. J.} {\bf 28} (1976), 601--612.

\bibitem{GH80} A. Gray and L. Hervella,
``The sixteen classes of almost Hermitian manifolds and their linear invariants",
{\it Ann. Mat. Pura Appl.} {\bf 123} (1980), 35--58.

\bibitem{H93} T. Higa,
``Weyl manifolds and Einstein-Weyl manifolds",
{\it Comm. Math. Univ. St. Pauli} {\bf 42} (1993), 143--160.

\bibitem{H94} T. Higa,
``Curvature tensors and curvature conditions in Weyl geometry",
{\it Comm. Math. Univ. St. Pauli} {\bf 43} (1994), 139--153.

\bibitem{K10} K.-D. Kirchberg, ``Eigenvalue estimates for the Dirac operator on K\"ahler-Einstein manifolds of even complex dimension",
{\it Ann. Global. Anal. Geom.} {\bf 38} (2010), 273--284. 

\bibitem{M09} Y. Matsuyama,
``Compact Einstein K\"ahler submanifolds of a complex projective space", {\it Balkan J. Geom Appl.} {\bf 14} (2009), 40--45.

\bibitem{PN91} P. Matzeu and S. Nik\v cevi\'c,
``Linear algebra of curvature tensors on Hermitian manifolds",
{\it An. Stiint. Nuni. Al. I. Cuza. Iasi Sect. I. a Mat.}
{\bf 37} (1991), 71--86.

\bibitem{N89} S. Nik\v cevi\'c,
``On the decomposition of curvature fields on Hermitian manifolds",
{\it Differential geometry and its applications (Eger, 1989)}, Colloq. Math. Soc. Janos Bolya {\bf 56},
North-Holland,
Amsterdam (1992), 555--568.

\bibitem{O10} V. Opriou, ``Hyper-K\"ahler structures on the tangent bundle of a K\"ahler manifold", {\it Balkan J. Geom. Appl.} {\bf 15} (2010),
104--119.

\bibitem{PSSW11} D. Phong, J. Song, J. Sturm, and B. Weinkove, ``On the convergence of the modified K\"ahler-Ricci flow and solitons",
{\it Comment. Math. Helv.} {\bf 86} (2011), 91-112.

\bibitem{SSV91} U. Simon, A. Schwenk-Schellschmidt, and H. Viesel,
``Introduction to the affine differential geometry of hypersurfaces",
{\it Lecture Notes, Science University of Tokyo} (1991).

\bibitem{ST69} I. Singer and J. Thorpe,
``The curvature of $4$-dimensional Einstein spaces"
{\it 1969 Global Analysis (Papers in Honor of K. Kodaira)}, Univ. Tokyo Press, Tokyo, 355--365.

\bibitem{S88} R. Strichartz,
``Linear algebra of curvature tensors and their covariant derivatives",
{\it Can. J. Math}, XL (1988), 1105--1143.

\bibitem{TV81}  F. Tricerri and  L. Vanhecke,
``Curvature tensors on almost Hermitian manifolds", {\it  Trans. Amer. Math. Soc.}
{\bf  267}  (1981), 365--397.

\bibitem{WZ11} X. Wang and B. Xhou, ``On the existence and non-existence of extremal metrics on toric K\"ahler surfaces",
Adv. Math. {\bf 226} (2011), 4429-4455.

\end{thebibliography}
\end{document}